\documentclass[leqno]{amsart}
\usepackage{dsfont}
\usepackage{mathrsfs}
\usepackage{amssymb}
\usepackage{amsmath}
\usepackage{amsfonts}
\usepackage{amsfonts,enumerate}
\usepackage{pifont}
\usepackage{enumerate}
\usepackage{graphics}
\usepackage{verbatim}
\usepackage{amsthm}
\usepackage{latexsym,bm}

\input xypic
\xyoption {all}

 \makeatletter
\numberwithin{equation}{section}
\newtheorem{theorem}{Theorem}[section]
\newtheorem{corollary}{Corollary}[section]

\newtheorem{lemma}{Lemma}[section]

\theoremstyle{definition}
\newtheorem{remark}{Remark}[section]

\newcommand{\real}{{\mathbb R}}

\newcommand{\M}{{\mathcal M}}
\newcommand{\N}{{\mathcal N}}

\newcommand{\BMO}{{\mathcal {BMO}}}
\newcommand{\MO}{{\mathcal {MO}}}

\newcommand{\8}{\infty}

\newcommand{\la}{\langle}
\newcommand{\ra}{\rangle}

\newcommand{\be}{\begin{eqnarray*}}
\newcommand{\ee}{\end{eqnarray*}}
\newcommand{\beq}{\begin{equation}}
\newcommand{\eeq}{\end{equation}}
\newcommand{\beqn}{\begin{equation*}}
\newcommand{\eeqn}{\end{equation*}}

\begin{document}

\title{Wavelet approach to operator-valued Hardy spaces}

\author{Guixiang Hong}
\address{Laboratoire de Math{\'e}matiques, Universit{\'e} de Franche-Comt{\'e}, \newline
25030 Besan\c{c}on Cedex, France\\ \emph{E-mail address:
guixiang.hong@univ-fcomte.fr}}

\author{Zhi Yin}
\address{Laboratoire de Math{\'e}matiques, Universit{\'e} de Franche-Comt{\'e}, \newline
25030 Besan\c{c}on Cedex, France\\ \emph{E-mail address:
hustyinzhi@163.com}}

\date{}
\maketitle

\begin{abstract}
This paper is devoted to the study of operator-valued Hardy spaces
via wavelet method. This approach is parallel to that in
noncommutative martingale case. We show that our Hardy spaces
defined by wavelet coincide with those introduced by Tao Mei via the
usual Lusin and Littlewood-Paley square functions. As a consequence,
we give an explicit complete unconditional basis of the Hardy space
$H_1(\mathbb{R})$ when $H_1(\mathbb{R})$ is equipped with an
appropriate operator space structure.
\end{abstract}

\section{Introduction}
In this paper, we exploit Meyer's wavelet methods to the study of
the operator-valued Hardy spaces. We are motivated by two rapidly
developed fields. The firs one is the theory of noncommutative
martingales inequalities. This theory had been already initiated in
the 1970's. Its modern period of development has begun with Pisier
and Xu's seminal paper \cite{PiXu97} in which the authors
established the noncommutative Burkholder-Gundy inequalities and
Fefferman duality theorem between $H_1$ and $BMO$. Since then many
classical results have been successfully transferred to the
noncommutative world (see \cite{JuXu03}, \cite{JuXu08},
\cite{Mei07}, \cite{BCPY}). In particular, motivated by
\cite{JLX06}, Mei \cite{Mei07} developed the theory of Hardy spaces
on $\mathbb{R}^n$ for operator-valued functions.

Our second motivation is the theory of wavelets founded by Meyer. It
is nowadays well known that this theory is important for many
domains, in particular in harmonic analysis. For instance, it
provides powerful tools to the theory of Calder\'{o}n-Zygmund
singular integral operators. More recently, Meyer's wavelet methods
were extended to study more sophistical subjects in harmonic
analysis. For example, the authors of \cite{FeLa02} exploited the
properties of Meyer's wavelets to give a characterization of product
$BMO$ by commutators; \cite{MPTT} deals with the estimates of
bi-parameter paraproducts.

It is in this spirit that we wish to understand how useful wavelet
methods are for noncommutative analysis. The most natural and
possible way would be first to do this in the semi-commutative case.
This is exactly the purpose of the present paper which could be
viewed as the first attempt towards the development of wavelet
techniques for noncommutative analysis.

A wavelet basis of $L_2(\real)$ is a complete orthonormal system
$(w_I)_{I\in\mathcal {D}}$, where $\mathcal {D}$ denotes the
collection of all dyadic intervals in $\mathbb{R}$, $w$ is a
Schwartz function satisfying the properties needed for Meryer's
construction in \cite{Mey90}, and
$$w_I(x)\doteq\frac{1}{|I|^{\frac{1}{2}}}w\big(\frac{x-c_I}{|I|}\big),$$
where $c_I$ is the center of $I$. The central facts that we will
need about the wavelet basis are the orthogonality between different
$w_I$'s, $\|w\|_{L_2(\real)}=1$ and the regularity of $w$,
$$\max(|w(x)|,|w'(x)|)\precsim(1+|x|)^{-m},\quad\forall m\geq2.$$

The analogy between wavelets and dyadic martingales is well known.
The key observation is the following parallelism:
$$\sum_{|I|=2^{-n+1}}\la
f,w_{I}\ra w_{I}\sim df_n,$$

where $df_n$ denotes $n$-th dyadic martingale difference of $f$. As
dyadic martingales are much easier to handle, this parallelism
explains why wavelet approach to many problems in harmonic analysis
is usually simple and efficient. On the other hand, it also
indicates that martingale methods may be used to deal with wavelets.
With this in mind, we develop the operator-valued Hardy spaces based
on the wavelet methods in the way which is well known in the
noncommutative martingales case. Then we show that our Hardy and BMO
spaces coincide with Mei's. In other words, we provide another
approach, which is much simpler than Mei's original one, to recover
all the results of \cite{Mei07}.

This paper is organized as follows. In section 1, we will give some
preliminaries on noncommutative analysis, the definition of
$\mathcal{H}_p(\real,\M)$ with $1\leq p<\infty$ and
$L_q\MO(\real,\M)$ with $2<q\leq\infty$ in our setting. In section
2, we are concerned with three duality results. The most important
one is the noncommutative analogue of the famous Fefferman duality
theorem between $\mathcal{H}^c_1(\real,\M)$ and
$\mathcal{BMO}^c(\real,\M)$. The second one is the duality between
$\mathcal{H}^c_p(\real,\M)$ and $L^c_{p'}\MO(\real,\M)$ with
$1<p<2$, where we need the noncommutative Doob's inequality, this is
why we consider the case $1<p<2$ independently. The last one is the
duality between $\mathcal{H}^c_p(\real,\M)$ and
$\mathcal{H}^c_{p'}(\real,\M)$ with $1<p<\infty$. As a corollary of
the last two results, we identify $\mathcal{H}^c_q(\real,\M)$ and
$L^c_q\MO(\real,\M)$ with $2<q<\infty$. Section 3 deals with the
interpolation of our Hardy spaces. In the last section, we show that
our Hardy spaces coincide with those of \cite{Mei07}. So, we can
give an explicit completely unconditional basis for the space
$H_1(\real)$, when $H_1(\real)$ is equipped with an appropriate
operator space structure.

We end this introduction by the convention that throughout the paper
the letter $c$ will denote an absolute positive constant, which may
vary from lines to lines, and $c_p$ a positive constant depending
only on $p$.

\section{Preliminaries}

\subsection{Operator-valued noncommutative $L_p$-spaces}
Let $\mathcal{M}$ be a von Neumann algebra equipped with a normal
semifinite faithful trace $\tau$ and $S^+_{\mathcal{M}}$ be the set
of all positive element $x$ in $\mathcal{M}$ with
$\tau(s(x))<\infty$, where $s(x)$ is the smallest projection $e$
such that $exe=x$. Let $S_{\mathcal{M}}$ be the linear span of
$S^+_{\mathcal{M}}$. Then any $x\in S_{\mathcal{M}}$ has finite
trace, and $S_{\mathcal{M}}$ is a $w^*$-dense $*$-subalgebra of
$\mathcal{M}$.

Let $1\leq p<\infty$. For any $x\in S_{\mathcal{M}}$, the operator
$|x|^p$ belongs to $S^+_{\mathcal{M}}$ ($|x|=(x^*x)^{\frac{1}{2}}$).
We define
$$\|x\|_p=\big(\tau(|x|^p)\big)^{\frac{1}{p}},\qquad\forall x\in S_{\mathcal{M}}.$$
One can check that $\|\cdot\|_p$ is well defined and is a norm on
$S_{\mathcal{M}}$. The completion of $(S_{\mathcal{M}},\|\cdot\|_p)$
is denoted by $L_p(\M)$ which is the usual noncommutative $L_p$-
space associated with $(\M,\tau)$. For convenience, we usually set
$L_{\infty}(\M)=\M$ equipped with the operator norm
$\|\cdot\|_{\M}$. The elements of $L_p(\M,\tau)$ can be described as
closed densely defined operators on $H$ ($H$ being the Hilbert space
on which $\M$ acts). We refer the reader to \cite{PiXu03} for more
information on noncommutative $L_p$-spaces.

In this paper, we are concerned with three operator-valued
noncommutative $L_p$-spaces. The first one is the Hilbert-valued
noncommutative space $L_p(\mathcal{M};H^c)$ (resp.
$L_p(\mathcal{M};H^r)$), which is studied at length in \cite{JLX06}.
For this space, we need the following properties. In the sequel,
$p'$ will always denote the conjugate index of $p$.

\begin{lemma}\label{duality between Lp(H)}
Let $1\leq p<\infty$. Then
\begin{equation}
(L_p(\mathcal{M};H^c))^*=L_{p'}(\mathcal{M};H^c).
\end{equation}
Thus, for $f\in L_p(\mathcal{M};H^c)$ and $g\in
L_{p'}(\mathcal{M};H^c)$, we have
$$|\tau(\la f,g\ra)|\leq\|f\|_{L_p(\mathcal{M};H^c)}\|g\|_{L_{p'}(\mathcal{M};H^c)},$$
where $\la,\ra$ denotes the inner product of $H$.
\end{lemma}

\begin{lemma}
Let $1\leq p_0<p<p_1\leq\8$, $0<\theta<1$,
$\frac{1}{p}=\frac{1-\theta}{p_0}+\frac{\theta}{p_1}$. Then
\begin{equation}\label{interpolation between Lp(H)}
[L_{p_0}(\mathcal{M};H^c),L_{p_1}(\mathcal{M};H^c)]_{\theta}=L_p(\mathcal{M};H^c).
\end{equation}
A same equality holds for row spaces.
\end{lemma}

The second one is the $\ell_{\8}$-valued noncommutative space
$L_p(\mathcal{M};\ell_{\8})$, which is studied by Pisier
\cite{Pis98} for an injective $\M$ and Junge \cite{Jun02} for a
general $\M$ (see also \cite{JuXu03} and \cite{JuXu06} for more
properties). About this one, we need the following property:

\begin{lemma}
Let $1\leq p<\8$. Then
$$(L_p(\mathcal{M};\ell_1))^*=L_{p'}(\mathcal{M};\ell_{\8}).$$
Thus, for $x=(x_n)_n\in L_p(\mathcal{M};\ell_1)$ and $y=(y_n)_n\in
L_{p'}(\mathcal{M};\ell_{\8})$, we have
\begin{equation}\label{duality between Lp(l1) and Lq(l8)}
\big|\sum_{n\geq1}\tau(x_ny_n)\big|\leq\|x\|_{L_p(\mathcal{M};\ell_1)}\|y\|_{L_{p'}(\mathcal{M};\ell_{\8})}.
\end{equation}
\end{lemma}

The third one is $L_p(\mathcal{M};\ell^c_{\8})$ for $2\leq
p\leq\infty$, which was introduced in \cite{DeJu04} and is related
with the second one by
$$\|(x_n)_n\|_{L_p(\mathcal{M};\ell^c_{\8})}=\|(|x_n|^2)_n\|_{L_{\frac{p}{2}}(\M;\ell_{\8})}.$$ And these are normed
spaces by the following characterization
$$\|(x_n)_n\|_{L_p(\mathcal{M};\ell^c_{\8})}=\inf_{x_n=y_na}\|(y_n)\|_{\ell_{\infty}(L_{\infty}(\M))}\|a\|_{L_p(\M)}.$$
We need the interpolation results about these spaces (see
\cite{Mus03}):

\begin{lemma}\label{interpolation between Lp(lc8)}
Let $2\leq p_0<p<p_1\leq\8$, $0<\theta<1$,
$\frac{1}{p}=\frac{1-\theta}{p_0}+\frac{\theta}{p_1}$. Then
\begin{equation}
[L_{p_0}(\mathcal{M};\ell^c_{\8}),L_{p_1}(\mathcal{M};\ell^c_{\8})]_{\theta}=L_p(\mathcal{M};\ell^c_{\8}).
\end{equation}
\end{lemma}

\subsection{Operator-valued Hardy spaces}
In this paper, for simplicity, we denote
$L_{\infty}(\real)\bar{\otimes}\M$ by $\mathcal{N}$. As indicated in
the introduction, one can observe that we have the following
operator-valued Calder\'on identity
\begin{equation}\label{Calderon identity}
f(x)=\sum_{I\in \mathcal {D}}\la f,w_I\ra w_I(x),
\end{equation}
which holds when $f\in L_2(\N)$. As in the classical case, for $f\in
S_{\N}$, we define the two Littlewood-Paley square functions as
\begin{equation}\label{column square function}
S_c(f)(x)=\Big(\sum_{I\in\mathcal {D}}\frac{|\la
f,w_I\ra|^2}{|I|}\mathds{1}_I(x) \Big)^{\frac{1}{2}}.
\end{equation}
\begin{equation}\label{row square function}
S_r(f)(x)=\Big(\sum_{I\in\mathcal {D}}\frac{|\la
f^*,w_I\ra|^2}{|I|}\mathds{1}_I(x) \Big)^{\frac{1}{2}}.
\end{equation}

For $1\leq p<\infty$, define
$$\|f\|_{\mathcal{H}^c_p}=\|S_c(f)\|_{L_p(\N)},$$
$$\|f\|_{\mathcal{H}^r_p}=\|S_r(f)\|_{L_p(\N)}.$$
These are norms, which can be seen easily from the space
$L_p(\mathcal{N};\ell^c_2(\mathcal{D}))$. So we define the spaces
$\mathcal{H}^c_p(\real,\M)$\;(resp. $\mathcal{H}^r_p(\real,\M)$) as
the completion of
$(S_{\mathcal{N}},\|\cdot\|_{\mathcal{H}^c_p(\real,\M)})$\;(resp.
$(S_{\mathcal{N}},\|\cdot\|_{\mathcal{H}^c_p(\real,\M)})$. Now, we
define the operator-valued Hardy spaces as follows: for $1\leq p<2,$
\begin{equation}\label{Hardy spaces for p<2}
\mathcal{H}_p(\real,\M)=\mathcal{H}^{c}_p(\real,\M)+\mathcal{H}^{r}_p(\real,\M)
\end{equation}
with the norm
$$\|f\|_{\mathcal{H}_p}=\inf\{\|g\|_{\mathcal{H}^{c}_p}+\|h\|_{\mathcal{H}^{r}_p}:
f=g+h,g\in\mathcal{H}^{c}_p,h\in\mathcal{H}^{r}_p\}$$ and for $2\leq
p<\infty$,
\begin{equation}\label{Hardy spaces for p>2}
\mathcal{H}_p(\real,\M)=\mathcal{H}^{c}_p(\real,\M)\cap\mathcal{H}^{r}_p(\real,\M)
\end{equation}
with the norm defined as
$$\|f\|_{\mathcal{H}_p}=\max\{\|f\|_{\mathcal{H}^{c}_p},\|f\|_{\mathcal{H}^{r}_p}\}.$$

We can identify $\mathcal{H}^c_p(\real,\M)$ as a subspace of
$L_p(\N;\ell^c_2(\mathcal {D}))$, which is related with the two maps
below.

\begin{lemma}
$\rm(i)$ The embedding map $\Phi$ is defined from
$\mathcal{H}^c_p(\real,\M)$ to $L_p(\N;\ell^c_2(\mathcal {D}))$ by
\begin{equation}\label{identify Hp as Lp(lc2)}
\Phi(f)=\Big(\frac{\la
f,w_I\ra}{|I|^{\frac{1}{2}}}\mathds{1}_{I}\Big)_{I\in\mathcal{D}}.
\end{equation}
$\rm(ii)$ The projection map $\Psi$ is defined from
$L_2(\N;\ell^c_2(\mathcal {D}))$ to $\mathcal{H}^c_2(\real,\M)$ by
\begin{equation}\label{projection from Lc2 as H2}
\Psi((g_I))=\sum_{I\in \mathcal
{D}}\int\frac{g_I}{|I|^{\frac{1}{2}}}\mathds{1}_{I}dy\cdot w_I.
\end{equation}
\end{lemma}

\subsection{Operator-valued $\mathcal{BMO}$ spaces}
For $\varphi\in L_{\infty}(\M; L^c_2(\real,\frac{dx}{1+x^2}))$, set
\begin{equation}
\|\varphi\|_{\BMO^c}=\sup_{J\in\mathcal
{D}}\Big\|\big(\frac{1}{|J|}\sum_{I\subset J}|\la
\varphi,w_I\ra|^2\big)^{\frac{1}{2}}\Big\|_{\M}
\end{equation}
and
$$\|\varphi\|_{\mathcal{BMO}^r}=\|\varphi^*\|_{\mathcal{BMO}^c(\real,\M)}.$$
These are again norms modulo constant functions. Define
$$\BMO^c(\real,\M)=\{\varphi \in L_{\infty}(\M;
L^c_2(\real,\frac{dx}{1+x^2})): \|\varphi\|_{\BMO^c}<\infty\}$$ and
$$\BMO^r(\real,\M)=\{\varphi \in L_{\infty}(\M;
L^r_2(\real,\frac{dx}{1+x^2})): \|\varphi\|_{\BMO^r}<\infty\}$$ Now
we define
$$\mathcal{BMO}(\real,\M)=\mathcal{BMO}^{c}(\real,\M)\cap\mathcal{BMO}^{r}(\real,\M).$$

As in the martingale case \cite{JuXu03}, we can also define
$L^c_p\MO(\real,\M)$ for all $2<p\leq\infty$. For $\varphi\in
L_p(\M; L^c_2(\real,\frac{dx}{1+x^2}))$, set
\begin{equation}
\|\varphi\|_{L^c_p\MO}=\Big\|(\frac{1}{|I^x_k|}\sum_{I\subset
I^x_k}|\la
\varphi,w_I\ra|^2)_k\Big\|_{L_{\frac{p}{2}}(\N;\ell_{\infty})}
\end{equation}
and
$$\|\varphi\|_{L^r_p\MO}=\|\varphi^*\|_{L^c_p\MO},$$
where $I^x_k$ denotes the unique dyadic interval with length
$2^{-k+1}$ that containing $x$. We will use the convention adopted
in \cite{JuXu06} for the norm in $L_{\frac{p}{2}}(\N;
\ell_{\infty}).$ Thus $$\Big\|(\frac{1}{|I^x_k|}\sum_{I\subset
I^x_k}|\la
\varphi,w_I\ra|^2)_k\Big\|_{L_{\frac{p}{2}}(\N;\ell_{\infty})}=\Big\|{\sup_k}^+\frac{1}{|I^x_k|}\sum_{I\subset
I^x_k}|\la \varphi,w_I\ra|^2\Big\|_{L_{\frac{p}{2}(\N)}}$$

These are norms, which can be seen from the Banach spaces
$L_p(\mathcal{N}\bar{\otimes}B(\ell_2(\mathcal{D}));\ell^c_{\infty})$.
Again, we can define
$$L^c_p\MO(\real,\M)=\{\varphi \in L_p(\M;
L^c_2(\real,\frac{dx}{1+x^2})): \|\varphi\|_{L^c_p\MO}<\infty\}$$
and
$$L^r_p\MO(\real,\M)=\{\varphi \in L_p(\M;
L^r_2(\real,\frac{dx}{1+x^2})): \|\varphi\|_{L^c_r\MO}<\infty\}$$
Define
$$L_p\MO(\real,\M)=L^c_p\MO(\real,\M)\cap L^r_p\MO(\real,\M).$$
Note that $L^c_{\infty}\MO(\real,\M)=\BMO^c(\real,\M)$. It is easy
to check all the spaces we defined here respect to the relevant
norms are Banach spaces.

\section{Duality}

To prove the first two duality results in this section, we need the
following noncommutative Doob's inequality from \cite{Jun02}.

Let $(\mathcal{E}_n)_n$ be the conditional expectation with respect
to a filtration $({\N}_n)_n$ of $\N$.

\begin{lemma}
Let $1<p\leq\infty$ and $f\in L_p(\N)$. Then
\begin{equation}\label{maximal inequality}
\|{\sup_n}^+\mathcal{E}_n(f)\|_{L_p(\N)}\leq c_p\|f\|_{L_p(\N)}.
\end{equation}
\end{lemma}

\begin{theorem}\label{duality between H1 and BMO}
We have
\begin{equation}
(\mathcal{H}^c_1(\real,\M))^*=\BMO^c(\real,\M)
\end{equation}
with equivalent norms. That is, every
$\varphi\in\mathcal{BMO}^c(\real,\M)$ induces a continuous linear
functional $l_{\varphi}$ on $\mathcal{H}^c_1(\real,\M)$ by
\begin{equation}\label{identity of duality between H1 and BMO}
l_{\varphi}(f)=\tau\int \varphi^*f,\quad\forall f\in
S_{\mathcal{N}}.
\end{equation}
Conversely, for every $l\in(\mathcal{H}^c_1(\real,\M))^*$, there
exits a $\varphi\in\mathcal{BMO}^c(\real,\M)$ such that
$l=l_{\varphi}$. Moreover,
$$c^{-1}\|\varphi\|_{\mathcal{BMO}^c}\leq\|l_{\varphi}\|_{(\mathcal{H}^c_1)^*}\leq c\|\varphi\|_{\mathcal{BMO}^c}$$
where $c>0$ is a universal constant.

Similarly, the duality holds between $\mathcal{H}^r_1$ and
$\mathcal{BMO}^r$, between $\mathcal{H}_1$ and $\mathcal{BMO}$ with
equivalent norms.
\end{theorem}

In order to adapt the arguments in the martingale case, we need to
define the truncated square functions for $n\in\mathbb{Z}$,
$$S_{c,n}(f)(x)=\Big(\sum^n_{k=-\infty}\sum_{|I|=2^{-k+1}}\frac{|\la
f,w_I\ra|^2}{|I|}\mathds{1}_I(x) \Big)^{\frac{1}{2}}.$$

\begin{proof} Since $S_{\mathcal{N}}$ is dense in
$\mathcal{H}^c_1(\real,\M)$, by an approximation argument, we only
need to prove the inequality
$$|l_{\varphi}(f)|\leq c\|\varphi\|_{\mathcal{BMO}^c}\|f\|_{\mathcal{H}^c_1}$$
for $f\in S_{\mathcal{N}}$. By approximation we may assume that
$S_{c,n}(f)(x)$ is invertible in $\M$ for all $x\in\real$ and
$n\in\mathbb{Z}$. Then we have
\begin{align*}
|l_{\varphi}(f)|&=|\tau\int \varphi^*fdx|\\
&=\Big|\sum_n\tau\int \sum_{|I|=2^{-n+1}}\la \varphi,w_I\ra^*w_{I}
\sum_{|I'|=2^{-n+1}}\la
f,w_{I'}\ra w_{I'}dx\Big|\\
&=\Big|\sum_n\tau\int \sum_{|I|=2^{-n+1}}\frac{\la
\varphi,w_I\ra^*}{|I|^{\frac{1}{2}}}\mathds{1}_{I}
\sum_{|I'|=2^{-n+1}}\frac{\la
f,w_{I'}\ra}{|I|^{\frac{1}{2}}}\mathds{1}_{I'}dx\Big|\\
&\leq\sum_n\Big(\tau\int\big|\sum_{|I|=2^{-n+1}}\frac{\la
f,w_{I}\ra}{|I|^{\frac{1}{2}}}\mathds{1}_{I}\big|^2S^{-1}_{c,n}(f)\Big)^{\frac{1}{2}}\\
&\qquad\cdot\Big(\tau\int\big|\sum_{|I|=2^{-n+1}}\frac{\la
\varphi,w_{I}\ra}{|I|^{\frac{1}{2}}}\mathds{1}_{I}\big|^2{S_{c,n}(f)}\Big)^{\frac{1}{2}}\\
&\leq\Big(\sum_n\tau\int\sum_{|I|=2^{-n+1}}\frac{|\la
f,w_I\ra|^2}{|I|}\mathds{1}_{I}S^{-1}_{c,n}(f)\Big)^{\frac{1}{2}}\\
&\qquad\cdot\Big(\sum_n\tau\int\sum_{|I|=2^{-n+1}}\frac{|\la
\varphi,w_I\ra|^2}{|I|}\mathds{1}_{I}{S_{c,n}(f)}\Big)^{\frac{1}{2}}\\
&=A\cdot B.
\end{align*}
In the above estimates, the first equality has used the
orthogonality of the $w_I$'s on different levels, the second one the
orthogonality of the $w_I$'s on the same level and the disjoint of
different dyadic $I$'s on the same level; the first inequality has
used the H\"{o}lder inequality in Lemma \ref{duality between Lp(H)},
and the second one the Cauchy-Schwarz inequality and the
disjointness of different $I$'s on the same level.

Now, let us estimate $A$: \be\begin{split}
A^2&=\sum_n\tau\int(S^2_{c,n}(f)-S^2_{c,n-1}(f))S^{-1}_{c,n}(f)\\
&=\sum_n\tau\int(S_{c,n}(f)-S_{c,n-1}(f))(1+S_{c,n-1}(f)S^{-1}_{c,n}(f))\\
&\leq\sum_n\tau\int(S_{c,n}(f)-S_{c,n-1}(f))\|1+S_{c,n-1}(f)S^{-1}_{c,n}(f)\|_{\8}\\
&\leq2\sum_n\tau\int(S_{c,n}(f)-S_{c,n-1}(f))\\
&=2\|f\|_{\mathcal{H}^c_1}.
\end{split}\ee
For the first inequality, we have used the H\"{o}lder inequality and
the positivity of $S_{c,n}(f)-S_{c,n-1}(f)$.

The second term is estimated as follows: \be\begin{split}
B^2&=\sum_{k}\tau\int(S_{c,k}(f)-S_{c,k-1}(f))\sum_{n\geq
k}\sum_{|I|=2^{-n+1}}\frac{|\la
\varphi,w_I\ra|^2}{|I|}\mathds{1}_{I}\\
&=\sum_{k}\tau\sum_j
(S_{c,k}(f)-S_{c,k-1}(f))\int_{I^j_k}\sum_{n\geq
k}\sum_{|I|=2^{-n+1}}\frac{|\la
\varphi,w_I\ra|^2}{|I|}\mathds{1}_{I}\\
&=\sum_{k}\tau
\sum_j\int_{I^j_k}(S_{c,k}(f)-S_{c,k-1}(f))\frac{1}{|I^j_k|}\sum_{I\subset
I^j_k}|\la
\varphi,w_I\ra|^2\\
&\leq\sum_{k}\sum_j\tau
\int_{I^j_k}(S_{c,k}(f)-S_{c,k-1}(f))\Big\|\frac{1}{|I^j_k|}\sum_{I\subset
I^j_k}|\la
\varphi,w_I\ra|^2\Big\|_{\8}\\
&\leq\|\varphi\|^2_{\BMO^c}\sum_{k}\sum_j\tau
\int_{I^j_k}(S_{c,k}(f)-S_{c,k-1}(f))\\
&=\|\varphi\|^2_{\BMO^c}\|f\|_{\mathcal{H}^c_1}\\
\end{split}\ee
The fist equality has used the Fubini theorem, the second one the
fact that $S_{c,k-1}(f)$ and $S_{c,k}(f)$ are constant on the dyadic
interval $I^j_k=[j2^{-k+1},(j+1)2^{-k+1})$; the first inequality has
used the H\"{o}lder inequality and the positivity of
$S_{c,n}(f)-S_{c,n-1}(f)$.

Now, let us begin to deal with another direction, i.e. suppose that
$l$ is a bounded linear functional on $\mathcal{H}^c_1(\real,\M)$,
we want to find an operator-valued function $\varphi$ in
$\BMO^c(\real,\M)$, such that $l=l_{\varphi}$ and
$l_{\varphi}(f)=\tau\int \varphi^*f$ for $f\in S_{\mathcal{\N}}$. By
the embedding operator $\Phi$ in (\ref{identify Hp as Lp(lc2)}) and
by the Banach-Hahn theorem, $l$ extends to a bounded continuous
functional on $L_1(\N;\ell^c_2(\mathcal {D}))$ of the same norm.
Then by the results in Lemma \ref{duality between Lp(H)} there
exists $g=(g_I)_{I\in \mathcal {D}}$ such that
$\|g\|_{L_{\8}(\N;\ell^c_2(\mathcal {D}))}=\|l\|$, and
$$l(f)=\tau\int\sum_{I\in\mathcal {D}}g^*_I\frac{\la f,w_I\ra}{|I|^{\frac{1}{2}}}\mathds{1}_{I},\quad\forall
f\in S_{\mathcal{N}}.$$

Now let $\varphi=\Psi(g)$, where $\Psi$ is defined as
(\ref{projection from Lc2 as H2}). The orthogonality of the $w_I$'s
yields \be\begin{split} \big\|\sum_{I\subset J}|\la
\varphi,w_I\ra|^2\big\|_{\M}&=\big\|\sum_{I\subset
J}|\int\frac{g_I}{|I|^{\frac{1}{2}}}\mathds{1}_{I}|^2\big\|_{\M}
\leq\big\|\sum_{I\subset J}\int_{J}|g_I|^2\big\|_{\M}\\
&\leq|J|\big\|\sum_{I\subset J}|g_I|^2\big\|_{L_{\8}(\N)}
\leq|J|\big\|(g_I)_I\big\|_{L_{\8}(\mathcal{N};\ell^c_2(\mathcal {D}))},\\
\end{split}\ee
where the first inequality has used the Kadison-Schwartz inequality.
Also thanks to the orthogonality of the $w_I$'s, we get
\be\begin{split} l(f)&=\tau\int\sum_{I\in\mathcal {D}}g^*_I\frac{\la
f,w_I\ra}{|I|^{\frac{1}{2}}}\mathds{1}_{I}
=\tau\int \varphi^*f\\
\end{split}\ee
for all $f\in S_{\mathcal{N}}$. Therefore, we complete the proof
about $\mathcal{H}^c_1(\real,\M)$ and $\mathcal{BMO}^c(\real,\M)$.
Passing to adjoint, we have the conclusion concerning
$\mathcal{H}^r_1$ and $\mathcal{BMO}^r$. Finally, by the classical
fact that the dual of a sum space is the intersection space, we
obtain the duality between $\mathcal{H}_1$ and $\mathcal{BMO}$.
\end{proof}

\begin{theorem}\label{duality between Hp and BMOp'}
Let $1<p<2$. We have
\begin{equation}
(\mathcal{H}^c_p(\real,\M))^*=L^c_{p'}\MO(\real,\M)
\end{equation}
with equivalent norms. That is, every $\varphi\in
L^c_{p'}\MO(\real,\M)$ induces a continuous linear functional
$l_{\varphi}$ on $\mathcal{H}^c_p(\real,\M)$ by
\begin{equation}\label{identity of duality between Hp and BMOp'}
l_{\varphi}(f)=\tau\int \varphi^*f,\quad\forall f\in
S_{\mathcal{N}}.
\end{equation}
Conversely, for every $l\in(\mathcal{H}^c_p(\real,\M))^*$, there
exists an operator-valued function $\varphi\in
L^c_{p'}\MO(\real,\M)$ such that $l=l_{\varphi}$ and
$$c_p^{-1}\|\varphi\|_{L^c_{p'}\MO}\leq\|l_{\varphi}\|_{(\mathcal{H}^c_p)^*}\leq \sqrt{2}\|\varphi\|_{L^c_{p'}\MO}$$

Similarly, the duality holds between $\mathcal{H}^r_p$ and
$L^r_{p'}$, between $\mathcal{H}_p$ and $L_{p'}\MO$ with equivalent
norms.
\end{theorem}

We need the following lemma of \cite{JuXu03}. We write it down for
the reader's convenience but without proof.

\begin{lemma}
Let $s,t$ be two real numbers such that $s<t$ and $0\leq s\leq1\leq
t\leq2$. Let $x,y$ be two positive operators such that $x\leq y$ and
$x^{t-s}, y^{t-s}\in L_1(\N)$. Then
$$\tau\int y^{-s/2}(y^t-x^t)y^{-s/2}\leq2\tau\int y^{-(s+1-t)/2}(y-x)y^{-(s+1-t)/2}.$$
\end{lemma}

\begin{proof} We need only to prove the first assertion on $\mathcal{H}^c_p$. Since $S_{\mathcal{N}}$ is dense in
$\mathcal{H}^c_p(\real,\M)$, by an approximation argument, we only
need to prove the inequality
$$|l_{\varphi}(f)|\leq c\|\varphi\|_{L^c_{p'}\MO}\|f\|_{\mathcal{H}^c_p}$$
for $f\in S_{\mathcal{N}}$. By approximation we may assume that
$S_{c,n}(f)(x)$ is invertible in $\M$ for all $x\in\real$ and
$n\in\mathbb{Z}$. By the similar principle in the noncommutative
martingale case as in \cite{JuXu03}, we have \be\begin{split}
|l_{\varphi}(f)|&=|\tau\int \varphi^*fdx|\\
&=\Big|\sum_n\tau\int \sum_{|I|=2^{-n+1}}\la\varphi,w_I\ra^*w_{I}
\sum_{|I'|=2^{-n+1}}\la
f,w_{I'}\ra w_{I'}dx\Big|\\
&=\Big|\sum_n\tau\int
\sum_{|I|=2^{-n+1}}\frac{\la\varphi,w_I\ra^*}{|I|^{\frac{1}{2}}}\mathds{1}_{I}
\sum_{|I'|=2^{-n+1}}\frac{\la
f,w_{I'}\ra}{|I|^{\frac{1}{2}}}\mathds{1}_{I'}dx\Big|\\
&\leq\sum_n\Big(\tau\int\big|\sum_{|I|=2^{-n+1}}\frac{\la
f,w_I\ra}{|I|^{\frac{1}{2}}}\mathds{1}_{I}\big|^2S^{p-2}_{c,n}(f)\Big)^{\frac{1}{2}}\\
&\qquad\cdot\Big(\tau\int\big|\sum_{|I|=2^{-n+1}}\frac{\la
\varphi,w_I\ra}{|I|^{\frac{1}{2}}}\mathds{1}_{I}\big|^2{S^{2-p}_{c,n}(f)}\Big)^{\frac{1}{2}}\\
&\leq\Big(\sum_n\tau\int\sum_{|I|=2^{-n+1}}\frac{|\la
f,w_I\ra|^2}{|I|}\mathds{1}_{I}S^{p-2}_{c,n}(f)\Big)^{\frac{1}{2}}\\
&\qquad\cdot\Big(\sum_n\tau\int\sum_{|I|=2^{-n+1}}\frac{|\la
\varphi,w_I\ra|^2}{|I|}\mathds{1}_{I}{S^{2-p}_{c,n}(f)}\Big)^{\frac{1}{2}}\\
&=A\cdot B.\\
\end{split}\ee

Now we need the above lemma to estimate the first term. Take $s=2-p$
and $t=2$, the lemma yields
\begin{align*}
A^2&=\sum_n\tau\int(S^2_{c,n}(f)-S^2_{c,n-1}(f))S^{p-2}_{c,n}(f)\\
&=\sum_n\tau\int S^{-(2-p)/2}_{c,n}(f)(S^2_{c,n}(f)-S^2_{c,n-1}(f))S^{-(2-p)/2}_{c,n}(f)\\
&\leq2\sum_n\tau\int
S^{-(1-p)/2}_{c,n}(f)(S_{c,n}(f)-S_{c,n-1}(f))S^{-(1-p)/2}_{c,n}(f)\\
&=2\sum_n\tau\int S_{c,n}(f)-S_{c,n-1}(f)S^{p-1}_{c,n}(f)\\
&\leq2\sum_n\tau\int S^p_{c,n}(f)-S^{p}_{c,n-1}(f)\\
&=2\|f\|^p_{\mathcal{H}^c_p}.
\end{align*}
The last inequality has used two elementary inequalities: $0\leq
S_{c,n-1}(f) \leq S_{c,n}(f)$ implies $S^{p-1}_{c,n-1}(f)\leq
S^{p-1}_{c,n}(f)$ for $0<p-1<1$; and
$\tau(S^{p}_{c,n-1}(f))\leq\tau(S^{\frac{1}{2}}_{c,n-1}(f)S^{p-1}_{c,n}(f)S^{\frac{1}{2}}_{c,n-1}(f)).$

The second term can be deduced from the nontrivial duality results
in Lemma \ref{duality between Lp(l1) and Lq(l8)} for $1<p<\infty$ as
follows.

\be\begin{split}  B^2&=\sum_{k}\tau\int
S^{2-p}_{c,k}(f)-S^{2-p}_{c,k-1}(f)\sum_{n\geq
k}\sum_{|I|=2^{-n+1}}\frac{|\la
\varphi,w_I\ra|^2}{|I|}\mathds{1}_{I}\\
&=\sum_{k}\tau
\sum_jS^{2-p}_{c,k}(f)-S^{2-p}_{c,k-1}(f)\int_{I^j_k}\sum_{n\geq
k}\sum_{|I|=2^{-n+1}}\frac{|\la
\varphi,w_{I}\ra|^2}{|I|}\mathds{1}_{I}\\
&=\sum_{k}\tau\sum_j
\int\mathds{1}_{I^j_k}(x)S^{2-p}_{c,k}(f)(x)-S^{2-p}_{c,k-1}(f)(x)\frac{1}{|I^j_k|}\sum_{I\subset
I^j_k}|\la
\varphi,w_{I}\ra|^2dx\\
&=\sum_{k}\tau \int
S^{2-p}_{c,k}(f)(x)-S^{2-p}_{c,k-1}(f)(x)\frac{1}{|I^x_k|}\sum_{I\subset
I^x_k}|\la
\varphi,w_{I}\ra|^2dx\\
&\leq\|\sum_{k}S^{2-p}_{c,k}(f)-S^{2-p}_{c,k-1}(f)\|_{L_{({p'}/2)'}}\Big\|\sup_k\frac{1}{|I^x_{k}|}
\sum_{I\subset I^x_{k}}|\la\varphi,w_{I}\ra|^2\Big\|_{L_{{p'}/2}}\\
&=\|\varphi\|^2_{L^c_{p'}\MO}\|f\|^{2-p}_{\mathcal{H}^c_p}
\end{split}\ee
The fist equality has used the Fubini theorem, the second one the
fact that $S_{c,k-1}(f)$ and $S_{c,k}(f)$ are constant on the dyadic
intervals with length $2^{-k+1}$.

For another direction, we can carry out the proof as that in the
case $p=1$. Suppose that $l$ is a bounded linear functional on
$\mathcal{H}^c_p(\real,\M)$. By the embedding operator $\Phi$ and by
Hahn-Banach theorem, and the results in Lemma \ref{duality between
Lp(H)}, we can find $g=(g_I)_{I\in \mathcal {D}}$ such that
$\|g\|_{L_{p'}(\N;\ell^c_2(\mathcal {D}))}=\|l\|$ and
$$l(f)=\tau\int\sum_{I\in\mathcal {D}}g^*_I\frac{\la f,w_I\ra}{|I|^{\frac{1}{2}}}\mathds{1}_{I},\forall
f\in S_{\mathcal{N}}.$$

Now let $\varphi=\Psi(g)$ defined in (\ref{projection from Lc2 as
H2}), the orthogonality of the $w_I$'s yields \be\begin{split}
\big\|{\sup_n}^+&\frac{1}{|I^x_{n}|}\sum_{I\subset I^x_n}|\la
\varphi,w_I\ra|^2\big\|_{L_{p'/2}(\N)}\\
&=\big\|{\sup_n}^+\frac{1}{|I^x_{n}|}\sum_{I\subset I^x_n}
|\int\frac{g_I}{|I|^{\frac{1}{2}}}\mathds{1}_{I}|^2\big\|_{L_{p'/2}(\N)}\\
&\leq\big\|{\sup_n}^+\frac{1}{|I^x_{n}|}\sum_{I\subset I^x_n}\int_{I^x_n}|g_I|^2\big\|_{L_{p'/2}(\N)}\\
&\leq\big\|{\sup_n}^+\frac{1}{|I^x_{n}|}\int_{I^x_n}\sum_{I\subset I^x_n}|g_I|^2\big\|_{L_{p'/2}(\N)}\\
&\leq\big\|{\sup_n}^+\frac{1}{|I^x_{n}|}\int_{I^x_n}\sum_{I\in\mathcal {D}}|g_I|^2\big\|_{L_{p'/2}(\N)}\\
&\leq c\big\|\sum_{I\in\mathcal {D}}|g_I|^2\big\|_{L_{p'/2}(\N)}\\
&=c\big\|(g_I)_{I}\big\|_{L_{p'}(\N;\ell^c_2(\mathcal {D}))},\\
\end{split}\ee
where for the first inequality we have used the Kadison-Schwartz
inequality, and the last inequality is (\ref{maximal inequality}).
Also due to the orthogonality of the $w_I$'s, we get
\be\begin{split} l(f)&=\tau\int\sum_{I\in\mathcal {D}}g^*_I\frac{\la
f,w_I\ra}{|I|^{\frac{1}{2}}}\mathds{1}_{I}
=\tau\int \varphi^*f,\\
\end{split}\ee
for all $f\in S_{\mathcal{N}}$. Therefore, we complete the proof
about $\mathcal{H}^c_p(\real,\M)$ and $L^c_{p'}\MO(\real,\M)$.
\end{proof}

Instead of using the noncommutative Doob's inequality, we will use
the following noncommutative Stein inequality from \cite{PiXu97} to
prove the duality between the spaces $\mathcal{H}^c_p$,
$1<p<\infty$.

Let $(\mathcal{E}_n)_n$ be the conditional expectation with respect
to a filtration $({\N}_n)_n$ of $\N$.

\begin{lemma}\label{Stein's inequality}
Let $1<p<\8$ and $a=(a_n)_n\in L_p(\N;\ell^c_2)$. Then there exists
a constant depending only on $p$ such that
\begin{equation}
\Big\|\big(\sum_{n}|\mathcal{E}_na_n|^2\big)^{\frac{1}{2}}\Big\|_{L_p(\N)}\leq
c_p\Big\|\big(\sum_{n}|a_n|^2\big)^{\frac{1}{2}}\Big\|_{L_p(\N)}.
\end{equation}
\end{lemma}

\begin{theorem}\label{duality between Hp}
For any $1<p<\infty$, we have
\begin{equation}
(\mathcal{H}^c_p(\real,\M))^*=\mathcal{H}^c_{p'}(\real,\M),
\end{equation}
\end{theorem}

\begin{proof}
By a similar reason as in the corresponding part of the proof of
Theorem \ref{duality between H1 and BMO}, we can carry out the
following calculation, \be\begin{split}
|l_{\varphi}(f)|&=|\tau\int \varphi^*fdx|\\
&=\Big|\sum_n\tau\int \sum_{|I|=2^{-n+1}}\la\varphi,w_I\ra^*w_{I}
\sum_{|I'|=2^{-n+1}}\la
f,w_{I'}\ra w_{I'}dx\Big|\\
&=\Big|\sum_n\tau\int
\sum_{|I|=2^{-n+1}}\frac{\la\varphi,w_I\ra^*}{|I|^{\frac{1}{2}}}\mathds{1}_{I}\frac{\la
f,w_{I}\ra}{|I|^{\frac{1}{2}}}\mathds{1}_{I}dx\Big|\\
&\leq\big\|\big(\sum_{I\in\mathcal {D} }\frac{|\la
f,w_{I}\ra|^2}{|I|}\mathds{1}_{I}\big)^{\frac{1}{2}}\big\|_{L_{p}(\real,\M)}
\cdot\big\|\big(\sum_{I\in\mathcal {D}}\frac{|\la
\varphi,w_{I}\ra|^2}{|I|}\mathds{1}_{I}\big)^{\frac{1}{2}}\big\|_{L_{p'}(\real,\M)}.\\
\end{split}
\ee

Now, we turn to the proof of the inverse direction. Take a bounded
linear functional $l\in(\mathcal{H}^c_p(\real,\M))^*$, by the
embedding operator $\Phi$ and the Hahn-Banach extension theorem, $l$
extends to a bounded linear functional on $L_p(\N;\ell^c_2)$ with
the same norm. Thus by (\ref{duality between Lp(H)}), there exists a
sequence $g=(g_I)_I$ such that
$$\|g\|_{L_{q}(\N;l^c_2(\mathcal {D}))}=\|l\|$$
and
$$l(f)=\tau\int\sum_{I\in \mathcal {D}}g^*_p\frac{\la f,w_I\ra}{|I|^{\frac{1}{2}}}\mathds{1}_{I},\forall
f\in S_{\mathcal{N}}.$$ Now let $\varphi=\Psi(g)$ where $\Psi$ is
defined in (\ref{projection from Lc2 as H2}), then applying the
Stein inequality (\ref{Stein's inequality}) to the conditional
expectation
$$\mathcal{E}_I(h)=\sum_{J}\frac{1}{|J|}\int_Jh(y)dy\cdot\mathds{1}_J,$$ where $J$ is dyadic interval with the same
length as $I$, we get \be\begin{split}
\|\varphi\|_{\mathcal{H}^c_{p'}(\real,\M)}
&=\|\big(\sum_{I\in\mathcal
{D}}|\frac{1}{|I|}\int_{I}g_Idy\cdot\mathds{1}_{I}|^2\big)^{\frac{1}{2}}\|
_{L_{p'}(\N)}\\
&\leq\|\big(\sum_{I\in\mathcal
{D}}|\mathcal{E}_I(g_I)|^2\big)^{\frac{1}{2}}\|
_{L_{p'}(\N)}\\
&\leq c_{p'}\|\big(\sum_{I\in\mathcal
{D}}|g_I|^2\big)^{\frac{1}{2}}\|
_{L_{p'}(\N)}.\\
\end{split}
\ee By the orthogonality of the $w_I$'s, we have \be\begin{split}
l(f)=\tau\int\sum_{I\in\mathcal {D}}g^*_I\frac{\la
f,w_I\ra}{|I|^{\frac{1}{2}}}\mathds{1}_{I}
=\tau\int \varphi^*f,\\
\end{split}\ee
for all $f\in S_{\mathcal{N}}$.
\end{proof}

From the proof of the second part of Theorem \ref{duality between H1
and BMO}, Theorem \ref{duality between Hp and BMOp'} and Theorem
\ref{duality between Hp}, we state the boundedness of $\Psi$ as a
corollary.

\begin{corollary}\label{boundedness of Psi}
$\rm(i)$ Let $1<p<\infty$, $\Psi$ is a projection map from
$L_p(\N;\ell^c_2(\mathcal{D}))$ onto $\mathcal{H}^c_p(\real,\M)$ if
we identify the latter as a subspace of the former.

$\rm(ii)$ Let $2< p\leq\infty$, $\Psi$ is also a bounded map from
$L_p(\N;\ell^c_2(\mathcal{D}))$ to $L^c_p\MO(\real,\M)$.
\end{corollary}

Theorem \ref{duality between Hp and BMOp'} and Theorem \ref{duality
between Hp} immediately imply the following corollary:

\begin{corollary}\label{Hp=BMOp} Let $2<p<\infty$. Then
$$\mathcal{H}^c_p(\real,\M)=L^c_p\MO(\real,\M),\quad \forall 2<p<\infty$$
with equivalent norms.
\end{corollary}

However, for the part
$L^c_p\MO(\real,\M)\subset\mathcal{H}^c_p(\real,\M)$, we can give
another proof. The idea is essentially similar to that in
\cite{Mei07}, the good news is that in our case, the argument seems
very elegant. Now we give the detailed proof.

\begin{proof}
Our tent space is defined as \be T^c_p=\Big\{f=\{f_I\}_I\in
L_p(\M;\ell^c_2(\mathcal {D})):\quad\tau\int\big(\sum_{I\in \mathcal
{D}}\frac{f^2_I}{|I|}\mathds{1}_{I}\big)^{\frac{p}{2}}<\infty\Big\}\ee
We claim that every $\varphi\in L^c_p\MO(\real,\M)$ induces a
bounded linear functional on $T^c_{p'}$, \be
l_{\varphi}(f)=\tau\int\sum_{I\in \mathcal
{D}}\frac{\la\varphi,w_I\ra^*}{|I|^{\frac{1}{2}}}\mathds{1}_{I}\frac{f_I}
{|I|^{\frac{1}{2}}}\mathds{1}_{I}dx\ee and
$\|l_{\varphi}\|\leq\|\varphi\|_{L^c_p\MO(\real,\M)}$. The proof is
just the copy of the proof of the first part in the last theorem.
Now $T^c_{p'}$ is naturally embedded into
$L_{p'}(\N;\ell^c_2(\mathcal {D}))$ by
$(f_I)_I\rightarrow(\frac{f_I}{|I|^{\frac{1}{2}}}\mathds{1}_{I})_I$.
So by the Hahn-Banach extension theorem, $l_{\varphi}$ extends to an
bounded linear functional on $L_{p'}(\N;\ell^c_2(\mathcal {D}))$
with the same norm. Then by the duality between
$$(L_{p'}(\N;\ell^c_2(\mathcal {D})))^*=L_{p}(\N;\ell^c_2(\mathcal {D})).$$
there exists a unique $h=(h_I)_I$ such that
$\|h\|_{L_{p}(\N;\ell^c_2(\mathcal {D}))}\leq\|l_{\varphi}\|$ and
for $f=(f_I)_I\in T^c_{p'}$, \be
l_{\varphi}(f)=\tau\int\sum_{I\in\mathcal {D}}h^*_I\frac{f_I}
{|I|^{\frac{1}{2}}}\mathds{1}_{I}dx\ee So we get
$$\frac{\la\varphi,w_I\ra}{|I|^{\frac{1}{2}}}\mathds{1}_{I}=h_I,$$
thus,
\be\begin{split}\|\varphi\|_{\mathcal{H}^c_p}&=\Big\|\big(\sum_{I\in
\mathcal
{D}}\frac{\la\varphi,w_I\ra^*}{|I|^{\frac{1}{2}}}\mathds{1}_{I}\big)^{\frac{1}{2}}\Big\|
_{L_p(\N)}\\
&=\|h_I\|_{L_{p}(\N;\ell^c_2(\mathcal {D}))}\leq\|l_{\varphi}\|
\end{split}\ee
\end{proof}

\section{Interpolation}
This section is devoted to the interpolation of our wavelet Hardy
spaces. The interpolation results below will be needed in the next
section to compare our Hardy spaces with those of Mei.

\begin{lemma}\label{interpolation between Hp}
Let $1<p_0<p<p_1<\infty$, we have
\begin{equation}
[\mathcal{H}^c_{p_0}(\real,\M),\mathcal{H}^c_{p_1}(\real,\M)]_{\theta}=\mathcal{H}^c_{p}(\real,\M)
\end{equation}
with equivalent norms, where $\theta$ satisfies
$\frac{1}{p}=\frac{1-\theta}{p_0}+\frac{\theta}{p_1}$.
\end{lemma}

\begin{proof} The embedding map $\Phi$ yields
$$[\mathcal{H}^c_{p_0},\mathcal{H}^c_{p_1}]_{\theta}\subset\mathcal{H}^c_p.$$
On the other hand, it is the boundedness of the projection map
$\Psi$ from  $L_p(\N;\ell^c_2(\mathcal {D}))$ to
$\mathcal{H}^c_p(\real,\M)$ stated in Corollary \ref{boundedness of
Psi} that yields the inverse direction.
\end{proof}

\begin{theorem}\label{interpolation between Hp and BMO}
Let $1\leq q<p<\infty$, we have
\begin{equation}
[\mathcal{BMO}^c(\real,\M),\mathcal{H}^c_{q}(\real,\M)]_{\frac{q}{p}}=\mathcal{H}^c_{p}(\real,\M)
\end{equation}
with equivalent norms.
\end{theorem}

\begin{proof} We will prove the theorem by a general strategy as appeared in \cite{Mus03}.

Step 1: We prove the conclusion for $2<q<p<\8$:
\begin{equation}\label{interpolation between Hp and BMO for p>2}
[\BMO^c(\real,\M),\mathcal{H}^c_q(\real,\M)]_{\frac{q}{p}}=\mathcal{H}^c_p(\real,\M).
\end{equation}
The identity  can be seen easily from the following two inclusions.
On one hand, the operator $\Phi$ which in (\ref{identify Hp as
Lp(lc2)}), together with (\ref{interpolation between Lp(H)}) yields
$$[\mathcal{H}^c_1(\real,\M),\mathcal{H}^c_{q'}(\real,\M)]_{\frac{q}{p}}\subset\mathcal{H}^c_{p'}(\real,\M).$$
Then by duality and Corollary \ref{Hp=BMOp}, we have
\begin{equation}
L^c_p\MO(\real,\M)\subset[\BMO^c(\real,\M),L^c_{q}\MO(\real,\M)]_{\frac{q}{p}}.
\end{equation}
On the other hand, the operator $\mathcal{T}$ identifying
$L^c_p\MO(\real,\M)$ as a subspace of
$L_p(L_{\8}(\mathcal{N}\bar{\otimes}B(\ell_2(\mathcal
{D}));\ell^c_{\8})$ defined by
\begin{equation}\label{identify BMOp as Lp(lc8)}
\mathcal{T}(\varphi)={\la
f,w_I\ra}{|I^t_k|^{-\frac{1}{2}}}\mathds{1}_{I\subset
I^t_k}(I)\otimes e_{I,1},
\end{equation} together with Lemma
\ref{interpolation between Lp(lc8)} yields
\begin{equation}
[\BMO^c(\real,\M),L^c_q\MO(\real,\M)]_{\frac{q}{p}}\subset
L^c_p\MO(\real,\M).
\end{equation}

Step 2: we prove the conclusion for $1<q<p<\8$. This step can be
divided into two substeps.

Substep 21: $p>2$. Let $p<s<\8$. By Step 1, we have
$$[\BMO^c(\real,\M),\mathcal{H}^c_p(\real,\M)]_{\frac{p}{s}}=\mathcal{H}^c_s(\real,\M).$$
On the other hand, by Theorem \ref{interpolation between Hp}, we
have
$$[\mathcal{H}^c_{q},\mathcal{H}^c_{s}]_{\theta}=\mathcal{H}^c_{p},$$
where(and in the rest of the paper) $\theta$ denote the
interpolation parameter. Then Wolff's interpolation theorem yields
the result.

Substep 22: $p\leq2$. Let $s>2$, then by Substep 21, we have
$$[\BMO^c(\real,\M),\mathcal{H}^c_p(\real,\M)]_{\frac{p}{s}}=\mathcal{H}^c_s(\real,\M).$$
Then together with Lemma \ref{interpolation between Hp}, Wolff's
interpolation theorem yields the result.

Step 3: we prove the conclusion for $1=q<p<\8$. Take $s>\max(p,2)$.
By Step 2 and duality \cite[Theorem 4.3.1]{BeLo76}, we get
$$[\mathcal{H}^c_{1},\mathcal{H}^c_{s}]_{\theta}=\mathcal{H}^c_{p}.$$
Then together with Step 2, Wolff's interpolation yields the
conclusion.
\end{proof}

\begin{remark}
If one can directly prove Lemma \ref{interpolation between Hp} for
$p_0=1$, we can prove the above theorem without the help of
$L^c_p\MO(\real,\M)$ for $2<p<\infty$ as carried out in \cite{BCPY},
where one needs an auxiliary space.
\end{remark}

\begin{theorem}\label{Birkhold-Gundy inequality}
For $1<p<\infty$, we have
$$\mathcal{H}_p(\real,\M)=L_p(\N)$$
with equivalent norms.
\end{theorem}

\begin{proof} There are several ways to prove this result. One can prove it by the
strategy in \cite{PiXu97} together with Stein's inequality
(\ref{Stein's inequality}). Here, we just use the fact that
$L_p(\M)$ with $1<p<\infty$ is a UMD space and our $(w_I)_I$ is an
complete orthonormal basis. So by Theorem 3.8 in \cite{HSV}, we have
$$\|f\|_{L_p(\N)}\simeq\Big(\mathbb{E}\Big\|\sum_{I\in\mathcal
{D}}\varepsilon_I\frac{\la f,w_I\ra}{|I|^{\frac{1}{2}}}
\mathds{1}_{I}\Big\|^p_{L_p(\N)}\Big)^{\frac{1}{p}}.$$ Then we
complete the proof for $2\leq p<\infty$ by Khintchine's
inequalities. Now, let us prove the case $1<p<2$. Let
$f\in\mathcal{H}_p(\real,\M)$, then for any $\epsilon>0$, by the
definition of $\mathcal{H}_p(\real,\M)$, there exists a
decomposition $f=f_c+f_r$ such that
$$\|f_c\|_{\mathcal{H}^c_p(\real,\M)}+\|f_r\|_{\mathcal{H}^r_p(\real,\M)}\leq\|f\|_{\mathcal{H}_p(\real,\M)}+\epsilon.$$
Take any $g\in L_{p'}(\mathcal{N})$, by the results for $p'>2$, the
operator-valued Calder\'{o}n identity (\ref{Calderon identity})
yields
\begin{align*}
|\tau\int gf^*|&=|\sum_{I\in\mathcal{D}}\tau\int \frac{\la
g,w_I\ra}{|I|^{\frac{1}{2}}}\mathds{1}_I\cdot\frac{\la
f^*,w_I\ra}{|I|^{\frac{1}{2}}}\mathds{1}_I|\\
&\leq|\sum_{I\in\mathcal{D}}\tau\int \frac{\la
g,w_I\ra}{|I|^{\frac{1}{2}}}\mathds{1}_I\cdot\frac{\la
f_c^*,w_I\ra}{|I|^{\frac{1}{2}}}\mathds{1}_I|\\
&\qquad+|\sum_{I\in\mathcal{D}}\tau\int \frac{\la
g,w_I\ra}{|I|^{\frac{1}{2}}}\mathds{1}_I\cdot\frac{\la
f_r^*,w_I\ra}{|I|^{\frac{1}{2}}}\mathds{1}_I|\\
&\leq\|S_c(g)\|_{L_{p'}(\mathcal{N})}\|S_c(f_c)\|_{L_{p}(\mathcal{N})}
+|S_r(g)\|_{L_{p'}(\mathcal{N})}\|S_r(f_r)\|_{L_{p}(\mathcal{N})}\\
&\leq
c_{p'}\|g\|_{L_{p'}}(\|f\|_{\mathcal{H}_p(\real,\M)}+\epsilon).
\end{align*}
Taking $\sup$ and let $\epsilon\rightarrow0$, we get the required
result.

Finally, we prove the inverse inequality. Let $f\in
L_p(\mathcal{N})$, by duality, we can find two sequences of
functions $(F_{c,I})_I\in L_p(\mathcal{N};\ell^c_2(\mathcal{D}))$
and $(F_{r,I})_I\in L_p(\mathcal{N};\ell^r_2(\mathcal{D}))$ such
that $F_{c,I}+F_{r,I}=\la f,w_I\ra|I|^{-\frac{1}{2}}\mathds{1}_I$
and
$$\|(F_{c,I})_I\|_{L_p(\mathcal{N};\ell^c_2(\mathcal{D}))}+\|(F_{r,I})_I\|_{L_p(\mathcal{N};\ell^r_2(\mathcal{D}))}
\leq\|f\|_{L_p(\mathcal{N})}.$$ Let $f_c=\Psi(({F_{c,I}})_I)$ and
$f_r=\Psi(({F_{r,I}})_I)$, by identity (\ref{Calderon identity}), we
have $f=f_c+f_r$. On the other hand, by the Stein inequality
(\ref{Stein's inequality}), we have
$\|f_c\|_{\mathcal{H}^c_p(\real,\M)}\leq
\|(F_{c,I})_I\|_{L_p(\mathcal{N};\ell^c_2(\mathcal{D}))}$ and
$\|f_r\|_{\mathcal{H}^r_p(\real,\M)}\leq
\|(F_{r,I})_I\|_{L_p(\mathcal{N};\ell^r_2(\mathcal{D}))}$. So we
have found the desired decomposition of $f$.
\end{proof}

\begin{theorem}\label{interpolation between Lp Hp and BMO}
The following results hold with equivalent norms:

$\rm(i)$ Let $1\leq q<p<\infty$, we have
\begin{equation}\label{interpolation between Lp and BMO}
[\mathcal{BMO}(\real,\M),L_q(\N)]_{\frac{q}{p}}=L_p(\N).
\end{equation}

$\rm(ii)$ Let $1<q<p\leq\infty$, we have
\begin{equation}\label{interpolation between H1 and Lp}
[\mathcal{H}_1(\real,\M),L_p(\N)]_{\frac{p'}{q'}}=L_q(\N).
\end{equation}

$\rm(iii)$ Let $1<p<\infty$, we have
\begin{equation}\label{interpolation between H1 and BMO}
[\mathcal{BMO}(\real,\M),\mathcal{H}_1(\real,\M)]_{\frac{1}{p}}=L_p(\N).
\end{equation}

\end{theorem}

In order to prove this theorem, we need the following result from
the theory of interpolation. We formulate it here without proof.

\begin{lemma}
Let $A_0,B_0,A_1,B_1$ be four Banach spaces satisfying the property
needed for interpolation. Then
$$[A_0+B_0,A_1+B_1]_{\theta}\supset[A_0,A_1]_{\theta}+[B_0,B_1]_{\theta}$$
and
$$[A_0\cap B_0,A_1\cap B_1]_{\theta}\subset[A_0,A_1]_{\theta}\cap[B_0,B_1]_{\theta}.$$
\end{lemma}

\begin{proof} $\rm(i)$ We also exploit the similar but different strategy with that in the proof of
Theorem \ref{interpolation between Hp and BMO}.

Step 1: we prove the results for $2\leq q<p<\infty$. By Theorem
\ref{Birkhold-Gundy inequality}, Theorem \ref{interpolation between
Hp and BMO} and the lemma, we have
$$[\mathcal{BMO}(\real,\M),L_q(\N)]_{\frac{q}{p}}\subset L_p(\N).$$
The inverse direction follows from $L_{\8}(\N)\subset
\mathcal{BMO}(\real,\M)$,
\be\begin{split}L_p(\N)&=[L_{\8}(\N),L_q(\N)]_{\frac{q}{p}}\\
&\subset[\mathcal{BMO}(\real,\M),L_q(\N)]_{\frac{q}{p}}\\
\end{split}\ee

Step 2: we prove the results for $1\leq q<2\leq p<\infty$. By Step
1, we have
$$[\mathcal{BMO}(\real,\M),L_2(\N)]_{\frac{2}{p}}=L_p(\N).$$
Together with
$$L_2(\N)=[L_p(\N),L_q(\N)]_{\theta},$$
Wolff's interpolation yields the conclusion.

Step 3: we prove the results for $1\leq q<p<2$. By Step 2, we have
$$[\mathcal{BMO}(\real,\M),L_p(\N)]_{\frac{p}{2}}=L_2(\N).$$
Together with
$$L_p(\N)=[L_2(\N),L_q(\N)]_{\theta},$$
Wolff's interpolation yields the conclusion.

$\rm(ii)$ The results for $1<q<p<\infty$ can be immediately proved
by duality and the partial results in $(i)$. For $p=\infty$, take
$q<s<\infty$, then by Wolff's argument, we get the conclusion.

$\rm(iii)$ First, we prove conclusion for $p<2$. Then by $\rm(i)$
and $\rm(ii)$, we have
$$[\mathcal{BMO}(\real,\M),L_p(\N)]_{\frac{p}{p'}}=L_{p'}(\N)$$
and
$$[\mathcal{H}_1(\real,\M),L_{p'}(\N)]_{\frac{p}{p'}}=L_p(\N).$$
Therefore, we end with Wolff's argument. Second, the proof for $p>2$
is the same. At last, when $p=2$, we can take $s>2$, by the results
for $p\neq2$ and reiteration theorem in \cite[Theorem
4.6.1]{BeLo76}, we get \be\begin{split} L_2&=[L_s,L_{s'}]_{\theta}
=[\mathcal{BMO}(\real,\M),\mathcal{H}_1(\real,\M)]_{\frac{1}{s}},
\mathcal{BMO}(\real,\M),\mathcal{H}_1(\real,\M)]_{\frac{1}{s'}}]_{\theta}\\
&=[\mathcal{BMO}(\real,\M),\mathcal{H}_1(\real,\M)]_{\theta}.\\
\end{split}
\ee
\end{proof}

\section{Comparison with Mei's results}
We denote the column Hardy space by $H^c_p(\real,\M)$ and the
bounded mean oscillation space by $BMO^c(\real,\M)$ in \cite{Mei07}.
We have the following result.

\begin{theorem} We have
$$\BMO^c(\real,\M)=BMO^c(\real,\M)$$ with equivalent norms. Similar
results holds for the row spaces. Consequently,
$\BMO(\real,\M)=\BMO(\real,\M)$ with equivalent norms.
\end{theorem}

The theorem can be easily seen from the corresponding
$BMO(\real,H)$-spaces. However, we can exploit the idea of
\cite{HSV} to prove our $\BMO^c(\real,\M)$ also coincide with that
defined by the mean oscillation $BMO(\real,H)$.

\begin{proof}
$\BMO^c(\real,\M)\subset BMO^c(\real,\M).$ Let $\varphi\in
\mathcal{BMO}_{c}(\mathbb{R},\mathcal{M})$. As in \cite{HSV}, fix a
finite interval $I\subset\mathbb{R}$, and consider the collections
of dyadic intervals

\begin{enumerate}[{\rm (1)}]
\item $\mathcal{D}_1:=\{J\in \mathcal{D};2|J|>|I|\}$'
\item $\mathcal{D}_2:=\{J\in \mathcal{D};2|J|\leq|I|,2J\cap2I=\emptyset\}$,
\item $\mathcal{D}_3:=\{J\in
\mathcal{D};2|J|\leq|I|,2J\cap2I\neq\emptyset\}$.
\end{enumerate}
Let $a_J=\langle\varphi,\omega_J\rangle$, then we have a priori
formal series
$$\varphi_1(x)=\sum_{J\in \mathcal{D}_1}a_J[\omega_J(x)-\omega_J(c_I)],  \varphi_i(x)=\sum_{J\in \mathcal{D}_i}a_J\omega_J(x), i=2,3,$$
where $c_I$ is the center of the interval $I$. Denote
$\varphi_I=\varphi_1+\varphi_2+\varphi_3$, by a similar discussion
in \cite{HSV}, we only need to prove:
$$\|\frac{1}{|I|}\int_I|\varphi_I(x)|^2dx\|_{\mathcal{M}}<\infty.$$
By scaling we can assume:
$$\sup_I\frac{1}{|I|}\|\sum_{J\subset I}|a_J|^2\|=1.$$
Then we have the obvious bound for individual terms
$\|a_J\|\leq|J|^{\frac{1}{2}}$.

Estimates for $\varphi_1$:
\begin{equation*}
\begin{split}
\|\frac{1}{|I|}\int_I|\varphi_1(x)|^2dx\|& \leq
\frac{1}{|I|}(\sum_{J\in\mathcal{D}_1}\|a_J\||\omega_J(x)-\omega_J(c_I)|)^2dx\\
&\leq c\frac{1}{|I|}\int_I[\sum_{J\in\mathcal{D}_1}|J|^{
\frac{1}{2}}|I||J|^{-\frac{3}{2}}(1+\frac{dist(I,J)}{|J|})^{-2}]^2dx\\
&=c[\sum_{j=0}^{\infty}\sum_{|J|\in(2^{j-1},2^j]|I|}|I||J|^{-1}(1+\frac{dist(I,J)}{|J|})^{-2}]^2<\infty.
\end{split}
\end{equation*}

Estimates for $\varphi_2$:
\begin{equation*}
\begin{split}
\|\frac{1}{|I|}\int_I|\varphi_2(x)|^2dx\|&\leq\frac{1}{|I|}\int_I\|\sum_{\mathcal{D}_2}a_J\omega_J(x)\|^2dx\\
&\leq
\frac{1}{|I|}\int_I(\sum_{\mathcal{D}_2}\|a_J\||\omega_J(x)|)^2dx\\
&\leq c\frac{1}{|I|}\int_I[\sum_{\mathcal{D}_2}|J|^{\frac{1}{2}}|J|^{-\frac{1}{2}}(\frac{dist(I,J)}{|J|})^{-2}]^2dx\\
&=c[\sum_{j=1}^{\infty}\sum_{|J|\in(2^{-j-1},2^{-j})|I|,
dist(I,J)>2^{-1}|I|}(\frac{dist(I,J)}{|J|})^{-2}]^2<\infty.
\end{split}
\end{equation*}

Estimates for $\varphi_3$:
\begin{equation*}
\begin{split}
\|\frac{1}{|I|}\int_I|\varphi_3(x)|^2dx\|&\leq\frac{1}{|I|}\|\sum_{J\in
\mathcal{D}_3}|a_J|^2\|\leq\frac{1}{|I|}\|\sum_{J\subset4I}|a_J|^2\|<\infty
\end{split}
\end{equation*}

Hence we deduce that:
$$\|\int_I|\varphi_I(x)|^2dx\|_{\mathcal{M}}\leq c\sum_{i=1}^3\|\int_I|\varphi_i(x)|^2dx\|_{\mathcal{M}}\leq c|I|$$

Now we turn to the proof of inverse direction
$BMO^c(\real,\M)\subset\BMO^c(\real,\M).$ Let $\varphi\in
BMO^{c}(\mathbb{R},\mathcal{M})$. The proof is very similar to that
in Mei's work \cite{Mei07}. For any dyadic interval
$I\subset\mathbb{R}$, write $\varphi=\varphi_1+\varphi_2+\varphi_3$,
where $\varphi_1=(\varphi-\varphi_{2I})\chi_{2I},
\varphi_2=(\varphi-\varphi_{2I})\chi_{2I^c},\varphi_3=\varphi_{2I}$.

Thus
\begin{equation*}
\begin{split}
\sum_{J\subset I}|\langle\varphi,\omega_J\rangle|^2\leq2
(\sum_{J\subset
I}|\langle\varphi_1,\omega_J\rangle|^2+\sum_{J\subset
I}|\langle\varphi_2,\omega_J\rangle|^2)
\end{split}
\end{equation*}

Estimates for $\varphi_1$:
\begin{equation*}
\begin{split}
\|\sum_{J\subset I}|\langle\varphi_1,\omega_J\rangle|^2\|\leq
\|\int|\varphi_1(x)|^2dx\|\leq
c\|\int_{2I}|\varphi-\varphi_{2I}|^2\|\leq c|I|
\end{split}
\end{equation*}

Estimates for $\varphi_2$:
\begin{equation*}
\begin{split}
\|\sum_{J\subset
I}|\langle\varphi_2,\omega_J\rangle|^2\|&=\|\sum_{J\subset
I}|\sum_{k=1}^{\infty}\int_{2^{k+1}I/2^kI}\varphi_2\omega_J dx|^2\|\\
&\leq\|\sum_{J \subset
I}(\sum_{k=1}^{\infty}\frac{1}{2^{2k}}\int_{2^{k+1}I/2^kI}|\varphi_2|^2)(\sum_{k=1}^{\infty}2^{2k}\int_{2^{k+1}I/2^kI}|\omega_J|^2)\|\\
&\leq
c(\sum_{k=1}^{\infty}\frac{1}{2^{2k}}\|\int_{2^{k+1}I}|\varphi-\varphi_{2I}|^2\|)\\
&\qquad\qquad(\sum_{J\subset
I}\sum_{k=1}^{\infty}2^{2k}\int_{2^{k+1}I/2^kI}|\omega_J|^2)\\
&\leq c|I|\|\varphi\|^2_{\BMO_c}\sum_{j=0}^{\infty}2^j
\sum_{k=1}^{\infty}\int_{2^{k+1}I/2^kI}2^{2k}\frac{|2^{-j}I|^3}{|2^{k}I|^4}\\
&\leq c|I|
\end{split}
\end{equation*}

Therefore $\|\sum_{J\subset
I}|\langle\varphi,\omega_J\rangle|^2\|\leq c|I|$, which completes
our proof.
\end{proof}

Combined with Theorem \ref{duality between Hp and BMOp'} and Theorem
\ref{interpolation between Hp and BMO}, we have the following
corollary

\begin{corollary} For $1\leq p<\infty$, we have
 $$\mathcal{H}^c_p(\real,\M)=H^c_p(\real,\M).$$ Similar results hold for
$\mathcal{H}^r_p$ and ${H}^{r}_p$, and $\mathcal{H}_p$ and ${H}_p$.
\end{corollary}

If $\M=\mathbb{C},$ $\mathcal{H}_1(\real, \mathbb{C})$ is just the
usual Hardy space $H_1(\real)$ on $\mathbb{R}.$ $H_1(\real)$ also
has the following characterization: $$H_1(\real)=\{f\in L_1(\real):
H(f)\in L_1(\real)\},$$ where $H$ is the Hilbert transform. For any
$f\in H_1(\real)$,
$$\|f\|_{H_1(\real)}\approx \|f\|_{L_1(\real)}+\|H(f)\|_{L_1(\real)}.$$
Thus $H_1(\real)$ can be viewed as a subspace of $L_1(\real)\oplus_1
L_1(\real)$. The latter direct sum has its natural operator
structure as an $L_1$ space. This induces an operator space
structure on $H_1(\real).$ Although $(w_I)_{I\in\mathcal{D}}$ is a
unconditional basis of $H_1(\real)$, Ricard \cite{Ric01} (see also
\cite{Ric02}) proved that $H_1(\real)$ does not have complete
unconditional basis. However, in noncommutative analysis, one can
introduce another natural operator space structure on $H_1(\real)$
as follows: $S_1(H_1(\real))=\mathcal{H}_1(\real, B(\ell_2)),$ where
$S_1$ is the trace class on $\ell_2.$ Then we have the following
result. Note that Ricard \cite{Ric02} obtained a similar result
using Hilbert space techniques.

\begin{corollary} The complete orthogonal systems $(w_I)_{I\in\mathcal {D}}$ of $L_2(\real)$ is
a completely unconditional basis for $H_1(\real)$ if we define the
operator space structure imposed on $H_1(\real)$ by
$\mathcal{S}_1(H_1(\real))=\mathcal{H}_1(\real,B(\ell_2))$.
\end{corollary}

\begin{proof}Fix a finite subset $\mathcal
{I}\subset\mathcal {D}$. Let
$T_{\varepsilon}f\doteq\sum_{I\in\mathcal {I}}\varepsilon_I\la
f,w_I\ra w_I$, where $\varepsilon_I=\pm1$. By the definition of
$\mathcal{H}^c_1(\real,\M)$, the orthogonality of
$(w_I)_{I\in\mathcal {D}}$  yields immediately that
\be\begin{split}\|T_{\varepsilon}f\|_{\mathcal{H}^c_1}&=\Big\|\Big(\sum_{I\in\mathcal
{I}}\frac{|\la f,w_I\ra|^2}{|I|}\mathds{1}_I(x)
\Big)^{\frac{1}{2}}\Big\|_{{L_1(\N)}}\\
&\leq\Big\|\Big(\sum_{I\in\mathcal {D}}\frac{|\la
f,w_I\ra|^2}{|I|}\mathds{1}_I(x)
\Big)^{\frac{1}{2}}\Big\|_{{L_1(\N)}}
=\|f\|_{\mathcal{H}^c_1}\\
\end{split}\ee
Similarly, the above inequality holds for
$\mathcal{H}^r_1(\real,\M)$. Now, let $f\in\mathcal{H}_1(\real,\M)$,
then for any $\epsilon>0$, there exists a decomposition $f=g+h$ such
that
$$\|g\|_{\mathcal{H}^c_1(\real,\M)}+\|h\|_{\mathcal{H}^r_1(\real,\M)}\leq\|f\|_{\mathcal{H}_1(\real,\M)}+\epsilon.$$
Therefore
\be\begin{split}\|T_{\varepsilon}f\|_{\mathcal{H}_1(\real,\M)}
&\leq\|T_{\varepsilon}g\|_{\mathcal{H}^c_1(\real,\M)}+
\|T_{\varepsilon}h\|_{\mathcal{H}^c_1(\real,\M)}\\
&\leq\|g\|_{\mathcal{H}^c_1(\real,\M)}+\|h\|_{\mathcal{H}^r_1(\real,\M)}\leq\|f\|_{\mathcal{H}_1(\real,\M)}+\epsilon.\\
\end{split}\ee
Let $\epsilon\rightarrow0$, we get the result.
\end{proof}

\end{document}